\documentclass[12pt]{amsart}
\usepackage{amscd,amssymb,amsthm,amsmath,amssymb,mathrsfs,enumerate}
\usepackage[matrix,arrow,curve]{xy}
\usepackage[margin=1.8cm]{geometry}

\newtheorem{theorem}{Theorem}

\newtheorem{lemma}[theorem]{Lemma}
\newtheorem*{lemma*}{Lemma}
\newtheorem{corollary}[theorem]{Corollary}

\newtheorem*{maintheorem*}{Main Theorem}
\newtheorem*{corollary*}{Corollary}

\theoremstyle{definition}
\newtheorem{example}[theorem]{Example}
\newtheorem*{example*}{Example}

\theoremstyle{remark}
\newtheorem{remark}[theorem]{Remark}

\makeatletter\@addtoreset{equation}{section} \makeatother

\newcommand{\mumu}{\boldsymbol{\mu}}

\author{Ivan Cheltsov and Piotr Pokora}

\title{On K-stability of $\mathbb{P}^3$ blown up along a~quintic elliptic curve}

\thanks{Throughout this paper, all varieties are assumed to be projective and defined over~$\mathbb{C}$.}

\begin{document}

\begin{abstract}
In this note, we study K-stability of smooth Fano threefolds that can be obtained by blowing up the~three-dimensional projective space
along a~smooth elliptic curve of degree five.
\end{abstract}


\address{ \emph{Ivan Cheltsov}\newline
\textnormal{University of Edinburgh, Edinburgh, Scotland
\newline
\texttt{i.cheltsov@ed.ac.uk}}}

\address{ \emph{Piotr Pokora}\newline \textnormal{Pedagogical University of Krakow, Krak\'ow, Poland
\newline
\texttt{piotrpkr@gmail.com}}}

\maketitle

Let $C_5$ be be a~smooth quintic elliptic curve in $\mathbb{P}^3$,
and let $\pi\colon X\to\mathbb{P}^3$ be the~blow up of this curve.
Then $X$ is a~smooth Fano threefold in the~family \textnumero 2.17,
and every smooth Fano 3-fold in this family can be obtained by
blowing up $\mathbb{P}^3$ along a~suitable smooth quintic elliptic curve.

It is well known that there exists the~following Sarkisov link:
\begin{equation}
\label{equation:diagram}\tag{$\bigstar$}
\xymatrix{
&X\ar@{->}[dl]_{\pi}\ar@{->}[dr]^{q}\\%
\mathbb{P}^3 && Q}
\end{equation}
where $Q$ is a~smooth quadric threefold in $\mathbb{P}^4$, and $q$ is a~blow up of a~smooth quintic elliptic curve~$C_5^\prime$.
Let $E_{\mathbb{P}^3}$ and $E_Q$ be the~exceptional divisors of $\pi$ and $q$, respectively.
If $\ell$ is a~general fiber of~the~natural projection $E_{\mathbb{P}^3}\to C_5$, then $\pi(\ell)$ is a~trisecant of the~quintic elliptic curve $C_5$.
Similarly, if $\ell^\prime$ is a~general fiber of the~projection~$E_Q\to C_5^\prime$, then $q(\ell^\prime)$ is a~secant of the~curve $C_5^\prime$
contained in $Q$.

\begin{example}
\label{example:book}
Let $\mathcal{E}$ be the~harmonic elliptic curve,
and let $\theta$ be an~element in  $\mathrm{Aut}(\mathcal{E})$ of order $4$ that fixes a~point $P\in C_5$.
Then it follows from \cite{Hulek} that
$$
\mathrm{Aut}\big(\mathcal{E},[5P]\big)\cong\mumu_5^2\rtimes\mumu_4,
$$
and there exists an~$\mathrm{Aut}(\mathcal{E},[5P])$-equivariant embedding $\phi\colon \mathcal{E}\hookrightarrow\mathbb{P}^4$
such that $\phi(\mathcal{E})$ is a~smooth quintic elliptic curve.
Let $G$ be a~subgroup in $\mathrm{Aut}(\mathcal{E},[5P])$ such that $G\cong\mumu_5\rtimes\mumu_4$.
Then $G$ fixes a~unique point in $\mathbb{P}^4$ that is not contains in the~hypersurface spanned by the~secants of the~quintic curve $\phi(\mathcal{E})$.
Let $\psi\colon \mathbb{P}^4\dasharrow\mathbb{P}^3$ be the~projection from this point.
Then $\psi\circ\phi(\mathcal{E})$ is a~smooth quintic elliptic curve.
Let $C_5=\psi\circ\phi(\mathcal{E})$. Then $\mathrm{Aut}(X)\cong\mumu_5\rtimes\mumu_4$,
and $X$ is K-stable \cite[Section~5.7]{Book}
\end{example}

Since being K-stable is an open condition, a~general member of the~family \textnumero 2.17 is K-stable.
In~fact, all smooth Fano threefolds in the~deformation family \textnumero 2.17  are expected to be K-stable \cite{Book}.
To show this it is enough to prove that $\beta(\mathbf{F})=A_X(\mathbf{F})-S_X(\mathbf{F})>0$
for every prime divisor $\mathbf{F}$ over the~Fano threefold $X$ \cite{Fujita2019,Li}, where $A_X(\mathbf{F})$ is the~log discrepancy of the~divisor $\mathbf{F}$, and
$$
S_X\big(\mathbf{F}\big)=\frac{1}{(-K_X)^3}\int\limits_0^{\infty}\mathrm{vol}\big(-K_X-u\mathbf{F}\big)du.
$$
Unfortunately, we are unable to prove this. Instead, we prove the~following weaker result:

\begin{maintheorem*}
Let $\mathbf{F}$~be~a~prime divisor over $X$ such that $\beta(\mathbf{F})\leqslant 0$, let $Z$ be its center on $X$.
Then $Z$ is a~point in $E_{\mathbb{P}^3}\cap E_Q$.
\end{maintheorem*}

By \cite[Corollary~4.14]{Zhuang}, the Main Theorem implies the~following corollary.

\begin{corollary}
\label{corollary:1}
Suppose that $\mathrm{Aut}(\mathbb{P}^3,C_5)$ does not fix a~point in $C_5$. Then $X$ is K-stable.
\end{corollary}

Observe that $\mathrm{Aut}(X)\cong\mathrm{Aut}(\mathbb{P}^3,C_5)$,
and all possibilities for the~group $\mathrm{Aut}(\mathbb{P}^3,C_5)$ can be easily derived from \cite{Hulek}.
Namely, if $C_5$ is general, then $\mathrm{Aut}(\mathbb{P}^3,C_5)$ is trivial, so Corollary~\ref{corollary:1} is not applicable.
If $\mathrm{Aut}(\mathbb{P}^3,C_5)$ is not trivial, then it must be isomorphic to one of the~following finite groups:
\begin{center}
$\mumu_5\rtimes\mumu_4$, $\mumu_5\rtimes\mumu_2$, $\mumu_6$, $\mumu_5$, $\mumu_4$, $\mumu_2$.
\end{center}
Furthermore, if $\mathrm{Aut}(\mathbb{P}^3,C_5)$  contains a~subgroup isomorphic to $\mumu_5$, it acts on $C_5$ by translations.
This implies that $\mathrm{Aut}(\mathbb{P}^3,C_5)$ does not fix a~point in $C_5$ $\iff$ $\mathrm{Aut}(\mathbb{P}^3,C_5)$  contains a~subgroup isomorphic to $\mumu_5$.
Therefore, Corollary~\ref{corollary:1} gives the~following generalization of Example~\ref{example:book}.

\begin{corollary}
\label{corollary:2}
Suppose that $\mathrm{Aut}(X)$ contains a~subgroup isomorphic to $\mumu_5$. Then $X$ is K-stable.
\end{corollary}

\begin{example}[{\cite{Hulek}}]
\label{example:1}
Fix $a\in\mathbb{C}$ such that $a\ne 0$ and $a^{10}+11a^5-1\ne 0$.
Let $C_5^\prime$ be the quintic elliptic curve in $\mathbb{P}^4$ given by the following system of equations:
$$
\left\{\aligned
&x_0^2+ax_2x_3-\frac{x_1x_4}{a}=0,\\
&x_1^2+ax_3x_4-\frac{x_2x_0}{a}=0,\\
&x_2^2+ax_4x_0-\frac{x_3x_1}{a}=0,\\
&x_3^2+ax_0x_1-\frac{x_4x_2}{a}=0,\\
&x_4^2+ax_1x_2-\frac{x_0x_3}{a}=0,
\endaligned
\right.
$$
where~$[x_0:x_1:x_2:x_3:x_4]$ are coordinates on~$\mathbb{P}^4$.
Let $\sigma$, $\tau$, $\iota$ be the~automorphisms of $\mathbb{P}^4$ given by
\begin{align*}
\sigma\big([x_0:x_1:x_2:x_3:x_4]\big)&=[x_1:x_2:x_3:x_4:x_0],\\
\tau\big([x_0:x_1:x_2:x_3:x_4]\big)&=[x_0:\omega_5x_1:\omega_5^2x_2:\omega_5^3x_3:\omega_5^4x_4],\\
\iota\big([x_0:x_1:x_2:x_3:x_4]\big)&=[x_0:x_4:x_3:x_2:x_1],
\end{align*}
where $\omega_5$ is a primitive fifth root of unity. Set $G=\langle\sigma,\tau,\iota\rangle$.
Then $G\cong\mumu_5^2\rtimes\mumu_2$, and $C_5^\prime$ is $G$-invariant.
Consider the following quadric hypersurface:
$$
Q=\Big\{x_0^2+ax_2x_3-\frac{x_1x_4}{a}=0\Big\}\subset\mathbb{P}^4.
$$
Observe that $Q$ is smooth, and $Q$ is $\langle\tau,\iota\rangle$-invariant.
Let $q\colon X\to Q$ be the blow up of the curve~$C_5^\prime$.
Then we have $\langle\tau,\iota\rangle$-equivariant Sarkisov link \eqref{equation:diagram} for an~appropriate non-singular quintic elliptic curve~$C_5\subset\mathbb{P}^3$.
Since $\langle\tau,\iota\rangle\cong\mumu_5\rtimes\mumu_2$, $X$ is K-stable by Corollary~\ref{corollary:2}.
\end{example}

Let $\Bbbk$ be a~subfield in $\mathbb{C}$ such that $C_5$ is defined over $\Bbbk$.
Then the~Sarkisov link \eqref{equation:diagram} is also defined over the~field $\Bbbk$.
Moreover, the Main Theorem and \cite[Corollary~4.14]{Zhuang} imply the~following result.

\begin{corollary}
\label{corollary:3}
If the intersection $E_{\mathbb{P}^3}\cap E_Q$ does not have $\Bbbk$-points, then $X$ is K-stable.
\end{corollary}

\begin{corollary}
\label{corollary:4}
If $C_5(\Bbbk)=\varnothing$ or $C_5^\prime(\Bbbk)=\varnothing$, then $X$ is K-stable.
\end{corollary}

In fact, one can show that $C_5(\Bbbk)=\varnothing$ if and only if $C_5^\prime(\Bbbk)=\varnothing$.

Corollary~\ref{corollary:4} has many applications.
For instance, if $\Bbbk$ is a~number field, there are infinitely many smooth quintic genus one curves in $\mathbb{P}^3$
defined over $\Bbbk$ that do not have $\Bbbk$-rational points \cite{ClarkLacy,Fisher2001,ShinderLin}.
Thus, using Corollary~\ref{corollary:4} and Pfaffian representations of quintic elliptic curves \cite{Fisher2010},
one can construct infinitely many explicit examples of K-stable smooth Fano threefolds in the family \textnumero 2.17.

\begin{example}[T.~Fisher]
\label{example:2}
Fix a prime $p\geqslant 2$. Let $C_5^\prime$ be the quintic elliptic curve in $\mathbb{P}^4$ given by
$$
\left\{\aligned
&x_0^2+px_1x_4-px_2x_3=0,\\
&x_1^2+x_0x_2-px_3x_4=0,\\
&x_2^2+x_1x_3-x_0x_4=0,\\
&px_3^2+px_2x_4-x_0x_1=0,\\
&px_4^2+x_0x_3-x_1x_2=0,
\endaligned
\right.
$$
let $Q$ be the quadric $\{x_0^2+px_1x_4-px_2x_3=0\}$, and let $q\colon X\to Q$ be the blow up along the curve~$C_5^\prime$,
where~$[x_0:x_1:x_2:x_3:x_4]$ are the coordinates on~$\mathbb{P}^4$.
Then \eqref{equation:diagram} exists for an~appropriate quintic elliptic curve~$C_5\subset\mathbb{P}^3$.
We can set $\Bbbk=\mathbb{Q}$. Then $C_5^\prime(\Bbbk)=\varnothing$,
so $X$ is K-stable by Corollary~\ref{corollary:4}.
\end{example}

Let us prove the Main Theorem. Let $\mathbf{F}$ be a prime divisor over~$X$,
and let $Z$ be its center on~$X$.
Suppose that $Z$ is \textbf{not} a point in $E_{\mathbb{P}^3}\cap E_Q$.
Let us show that $\beta(\mathbf{F})>0$.

If $Z$ is a surface, then it follows from \cite{Fujita2016} that $\beta(\mathbf{F})>0$.
Thus, we may assume that
\begin{itemize}
\item either $Z$ is a point,
\item or $Z$ is an irreducible curve.
\end{itemize}

Let $P$ be any point in $Z$. Choose an~irreducible smooth surface $S\subset X$ such that $P\in S$. Set
$$
\tau=\mathrm{sup}\Big\{u\in\mathbb{Q}_{\geqslant 0}\ \big\vert\ \text{the divisor  $-K_X-uS$ is pseudo-effective}\Big\}.
$$
For~$u\in[0,\tau]$, let $P(u)$ be the~positive part of the~Zariski decomposition of the~divisor $-K_X-uS$,
and let $N(u)$ be its negative part. Then $\beta(S)=1-S_X(S)$, where
$$
S_X(S)=\frac{1}{-K_X^3}\int\limits_{0}^{\infty}\mathrm{vol}\big(-K_X-uS\big)du=\frac{1}{24}\int\limits_{0}^{\tau}P(u)^3du.
$$
Let us show how to compute $P(u)$ and $N(u)$. Set $H_{\mathbb{P}^3}=\pi^*(\mathcal{O}_{\mathbb{P}^3}(1))$ and $H_{Q}=q^*(\mathcal{O}_{Q}(1))$. Then
\begin{align*}
H_{\mathbb{P}^3}&\sim 2H_{Q}-E_{Q}, & H_{Q}&\sim 3H_{\mathbb{P}^3}-E_{\mathbb{P}^3},\\
E_{\mathbb{P}^3}&\sim 5H_{Q}-3E_{Q}, & E_{Q}&\sim 5H_{\mathbb{P}^3}-2E_{\mathbb{P}^3}.
\end{align*}
Let us compute $P(u)$ and $N(u)$ in the~following cases: $S\in |H_{\mathbb{P}^3}|$, $S\in|H_{Q}|$, and $S=E_{\mathbb{P}^3}$.

\begin{example}
\label{example:AZ-surfaces-HP3}
Suppose that $S\in |H_{\mathbb{P}^3}|$. Then $\tau=\frac{3}{2}$, since $-K_{X}-uS\sim_{\mathbb{R}}\frac{3-2u}{2}S+\frac{1}{2}E_{Q}$.
Based on that, the~positive part of the~Zariski decomposition has the~following form
$$
P(u)\sim_{\mathbb{R}} \left\{\aligned
& (4-u)H_{\mathbb{P}^3}-E_{\mathbb{P}^3} \ \text{ for } 0\leqslant u\leqslant 1, \\
& (3-2u)H_{Q}\ \text{ for } 1\leqslant u\leqslant \frac{3}{2},
\endaligned
\right.
$$
and the~negative part
$$
N(u)= \left\{\aligned
&0\ \text{ for } 0\leqslant u\leqslant 1, \\
&(u-1)E_{Q}\ \text{ for } 1\leqslant u\leqslant \frac{3}{2},
\endaligned
\right.
$$
which gives
$$
S_{X}(S)=\frac{1}{24}\int\limits_{0}^{\frac{3}{2}}\big(P(u)\big)^3du=\frac{1}{24}\int\limits_0^{1}24-u^3+12u^2-33udu+\frac{1}{24}\int\limits_{1}^{\frac{3}{2}}2(3-2u)^3du=\frac{23}{48}.
$$
\end{example}

\begin{example}
\label{example:AZ-surfaces-HPQ}
Suppose that $S\in |H_{Q}|$. Then $-K_{X}-uS\sim_{\mathbb{R}} \frac{4-3u}{3}S+\frac{1}{3}E_{\mathbb{P}^3}$.
Then $\tau=\frac{4}{3}$,
$$
P(u)\sim_{\mathbb{R}} \left\{\aligned
& (3-u)H_Q-E_Q \ \text{ for } 0\leqslant u\leqslant 1, \\
& (4-3u)H_{\mathbb{P}^3}\ \text{ for } 1\leqslant u\leqslant \frac{4}{3},
\endaligned
\right.
$$
and
$$
N(u)= \left\{\aligned
&0\ \text{ for } 0\leqslant u\leqslant 1, \\
&(u-1)E_{\mathbb{P}^3}\ \text{ for } 1\leqslant u\leqslant \frac{4}{3},
\endaligned
\right.
$$
which gives
$$
S_{X}(S)=\frac{1}{24}\int\limits_0^{1}24-2u^3+18u^2-39udu+\frac{1}{24}\int\limits_{1}^{\frac{4}{3}}(4-3u)^3du=\frac{121}{288}.
$$
\end{example}

\begin{example}
\label{example:AZ-surfaces-E}
Suppose that $S=E$. Then $-K_{X}-uS\sim_{\mathbb{R}} \frac{3-5u}{5}E_{\mathbb{P}^{3}}+\frac{4}{5}E_{Q}$.
Then $\tau=\frac{3}{5}$,
$$
P(u)\sim_{\mathbb{R}} \left\{\aligned
&4H_{\mathbb{P}^3}-(1+u)E_{\mathbb{P}^3} \ \text{ for } 0\leqslant u\leqslant \frac{1}{3}, \\
&(3-5u)H_{Q}\ \text{ for } \frac{1}{3}\leqslant u\leqslant \frac{3}{5},
\endaligned
\right.
$$
and
$$
N(u)= \left\{\aligned
&0\ \text{ for } 0\leqslant u\leqslant \frac{1}{3}, \\
&(3u-1)E_{Q}\ \text{ for } \frac{1}{3}\leqslant u\leqslant \frac{3}{5},
\endaligned
\right.
$$
which gives
$$
S_{X}(S)=\frac{1}{24}\int\limits_0^{\frac{1}{3}}20u^3-60u+24du+\frac{1}{24}\int\limits_{\frac{1}{3}}^{\frac{3}{5}}2(3-5u)^3du=\frac{227}{1080}.
$$
\end{example}

Now, we choose an~irreducible curve $C\subset S$ that contains the~point $P$.
For instance, if $Z$ is a~curve, and $S$ contains $Z$, then we can choose $C=Z$.
Since $S\not\subset\mathrm{Supp}(N(u))$,  we can write
$$
N(u)\big\vert_S=d(u)C+N^\prime(u),
$$
where $N^\prime(u)$ is an~effective $\mathbb{R}$-divisor on $S$ such that $C\not\subset\mathrm{Supp}(N^\prime(u))$, and $d(u)=\mathrm{ord}_C(N(u)\vert_S)$.
Now, for every $u\in [0,\tau]$, we set
$$
t(u)=\sup\Big\{v\in \mathbb R_{\geqslant 0} \ \big|\ \text{the divisor $P(u)\big|_S-vC$ is pseudo-effective}\Big\}.
$$
For $v\in [0, t(u)]$, we let $P(u,v)$ be the~positive part of the~Zariski decomposition of $P(u)|_S-vC$,
and we let $N(u,v)$ be its negative part. Following \cite{AbbanZhuang,Book}, we let
$$
S\big(W^S_{\bullet,\bullet};C\big)=\frac{3}{(-K_X)^3}\int\limits_0^{\tau}d(u)\Big(P(u)\big\vert_{S}\Big)^2du+\frac{3}{(-K_X)^3}\int\limits_0^\tau\int\limits_0^{\infty}\mathrm{vol}\big(P(u)\big\vert_{S}-vC\big)dvdu,
$$
which we can rewrite as
$$
S\big(W^S_{\bullet,\bullet};C\big)=\frac{3}{(-K_X)^3}\int\limits_0^{\tau}d(u)\big(P(u,0)\big)^2du+\frac{3}{(-K_X)^3}\int\limits_0^\tau\int\limits_0^{t(u)}\big(P(u,v)\big)^2dvdu.
$$
If $Z$ is a~curve, $Z\subset S$ and $C=Z$, then it follows from \cite{AbbanZhuang,Book} that
\begin{equation}
\label{equation:Kento-curve}
\frac{A_X(\mathbf{F})}{S_X(\mathbf{F})}\geqslant\min\Bigg\{\frac{1}{S_X(S)},\frac{1}{S\big(W^S_{\bullet,\bullet};C\big)}\Bigg\}.
\end{equation}
Hence, if $Z$ is a~curve, $Z\subset S$, $C=Z$ and $S(W^S_{\bullet,\bullet};C)<1$, then $\beta(\mathbf{F})>0$, since $S_X(S)<1$ by \cite{Fujita2016}.

\begin{lemma}
\label{lemma:E}
Suppose that $Z$ is a~curve, $Z\subset E_{\mathbb{P}^3}$, and $\pi(Z)$ is not a~point. Then $\beta(\mathbf{F})>0$.
\end{lemma}

\begin{proof}
Let $e$ be the~invariant of the~ruled surface $E_{\mathbb{P}^3}$ defined in Proposition 2.8 in \cite[Chapter V]{Hartshorne}.
Then $e\geqslant -1$ by \cite{Nagata}. Moreover, there is a~section $C_{0}$ of the~projection $E_{\mathbb{P}^3}\to C_5$ such that $C_0^2=-e$.
Let $\ell$ a~fiber of this projection.
Then
$H_{\mathbb{P}^{3}}\vert_{E_{\mathbb{P}^3}}\equiv 5\ell$ and $E_{\mathbb{P}^3}\vert_{E_{\mathbb{P}^3}}\equiv -C_{0} + \lambda \ell$
for some integer $\lambda$. Since
$$
-20=-c_{1}\big(N_{C_5/\mathbb{P}^{3}}\big)=E_{\mathbb{P}^3}^{3}=(-C_{0}+\lambda \ell)^{2}=-e -2\lambda,
$$
we get $\lambda=\frac{20-e}{2}$. 
Then $e$ is even, so $e\geqslant 0$.
Moreover, since $3H_{\mathbb{P}^3}-E_{\mathbb{P}^3}\sim H_{Q}$ is nef,
the divisor
$$
\big(3H_{\mathbb{P}^{3}}-E_{\mathbb{P}^{3}}\big)\big\vert_{E_{\mathbb{P}^{3}}}\equiv C_{0}+(15-\lambda)\ell
$$
is also nef. Then
$0\leqslant \big( C_{0}+(15-\lambda)\ell\big)\cdot C_{0} = 15 - e - \lambda = 15 - e - \frac{20-e}{2}$,
which implies $e\leqslant 10$ and hence we have $e\in\{0,2,4,6,8,10\}$.

Now, we set $S=E_{\mathbb{P}^{3}}$ and $C=Z$.
Using \eqref{equation:Kento-curve}, we see that to prove that $\beta(\mathbf{F})>0$,  it is enough to show that $S(W^S_{\bullet,\bullet};C)<1$.
Let us estimate $S(W^S_{\bullet,\bullet};C)$.
It follows from Example~\ref{example:AZ-surfaces-E} that $\tau=\frac{3}{5}$ and
$$
P(u)\big\vert_{S}\equiv\left\{\aligned
& (1+u)C_{0}+\bigg(10+\frac{1}{2}e+\frac{1}{2}ue-10u\bigg)\ell \ \text{ for } 0\leqslant u \leqslant \frac{1}{3}, \\
& (3-5u)C_{0}+\bigg(15+\frac{3}{2}e-25u- \frac{5}{2}ue \bigg)\ell \ \text{ for } \frac{1}{3} \leqslant u\leqslant \frac{3}{5}.
\endaligned
\right.
$$
Moreover, if $0\leqslant u \leqslant \frac{1}{3}$, then $N(u)=0$.
Furthermore, if $\frac{1}{3} \leqslant u\leqslant \frac{3}{5}$, then
$$
N(u)\big\vert_{S}=(3u-1)E_Q\big\vert_{S},
$$
and $E_Q\big\vert_{S}\equiv 2C_0+(5+e)\ell$.
But it follows from Proposition 2.20 in \cite[Chapter~V]{Hartshorne} that
$$
Z\equiv aC_{0} + b\ell
$$
for some integers $a$ and $b$  such that $a\geqslant 0$ and $b\geqslant ae$.
Since $\pi(Z)$ is not a~point, we also have $a\geqslant 1$.  This gives $\mathrm{ord}_{C}(E_Q\vert_{S})\leqslant 2$.
Hence, if $\frac{1}{3} \leqslant u\leqslant \frac{3}{5}$, then $d(u)\leqslant 2(3u-1)$.
This gives
\begin{multline*}
S(W_{\bullet,\bullet}^{{S}};{C})=
\frac{3}{24}\int\limits_{\frac{1}{3}}^{\frac{3}{5}}d(u)\Big(P(u)\big\vert_{S}\Big)^2du+
\frac{3}{24}\int\limits_0^{\frac{3}{5}}\int\limits_0^\infty \mathrm{vol}\big(P(u)\big\vert_{{S}}-v{C}\big)dvdu\leqslant \\
\leqslant\frac{3}{24}\int\limits_{\frac{1}{3}}^{\frac{3}{5}}2(3u-1)\Big(P(u)\big\vert_{S}\Big)^2du+
\frac{3}{24}\int\limits_0^{\frac{3}{5}}\int\limits_0^\infty \mathrm{vol}\big(P(u)\big\vert_{{S}}-v{C}\big)dvdu = \\
=\frac{3}{24}\int\limits_{\frac{1}{3}}^{\frac{3}{5}}2(3u-1)(250u^2-300u+90)du+
\frac{3}{24}\int\limits_0^{\frac{3}{5}}\int\limits_0^\infty \mathrm{vol}\big(P(u)\big\vert_{{S}}-v{C}\big)dvdu = \\
=\frac{32}{405}+\frac{3}{24}\int\limits_0^{\frac{3}{5}}\int\limits_0^\infty \mathrm{vol}\big(P(u)\big\vert_{{S}}-v{C}\big)dvdu
=\frac{32}{405}+\frac{3}{24}\int\limits_0^{\frac{3}{5}}\int\limits_0^\infty \mathrm{vol}\big(P(u)\big\vert_{{S}}-v(aC_{0} + b\ell)\big)dvdu.
\end{multline*}
On the~other hand, since $a\geqslant 1$, we have
$$
\frac{3}{24}\int\limits_0^{\frac{3}{5}}\int\limits_0^\infty \mathrm{vol}\big(P(u)\big\vert_{{S}}-v(aC_{0}+b\ell)\big)dvdu\leqslant\frac{3}{24}\int\limits_0^{\frac{3}{5}}\int\limits_0^\infty \mathrm{vol}\big(P(u)\big\vert_{{S}}-vC_{0}\big)dvdu,
$$
Therefore, to show that $S(W^S_{\bullet,\bullet};C)<1$, we may assume that $Z=C_0$. Then
$$
t(u) =\left\{\aligned
& 1+u\ \text{ for } 0\leqslant u\leqslant \frac{1}{3}, \\
& 3-5u\ \text{ for } \frac{1}{3}\leqslant u\leqslant \frac{3}{5}.
\endaligned
\right.
$$
Moreover, if $0 \leqslant u \leqslant \frac{1}{3}$ and $v\in[0,t(u)]$, then
$$
P(u,v) =(1+u-v)C_{0}+ \bigg(10+\frac{1}{2}e+\frac{1}{2}ue-10u \bigg)\ell
$$
and the~negative part $N(u,v)$ is trivial.
Similarly, if $\frac{1}{3}\leqslant u \leqslant \frac{3}{5}$ and  $v\in[0,t(u)]$, then
$$
P(u,v)=(3-5u-v)C_{0}+\bigg(15+\frac{3}{2}e-25u- \frac{5}{2}ue \bigg)\ell
$$
and the~negative part $N(u,v)$ is trivial. Using the~collected data, we compute
\begin{multline*}
\frac{3}{24}\int\limits_0^{\frac{3}{5}}\int\limits_0^\infty \mathrm{vol}\big(P(u)\big\vert_{{S}}-vC_{0}\big)dvdu=\frac{3}{24}\int\limits_0^{\frac{3}{5}}\int\limits_0^{t(u)}\big(P(u,v)\big)^2dvdu=\\
=\frac{3}{24}\int\limits_{0}^{\frac{1}{3}}\int\limits_{0}^{u+1}\bigg(20+(e-20)v-20u^2-ev^2+(e+20)vu\bigg)dvdu+\\
+\frac{3}{24}\int\limits_{\frac{1}{3}}^{\frac{3}{5}}\int\limits_{0}^{3-5u}\bigg( 90-300u+(3e-30)v+250u^2-ev^2 +(-5e+50)vu\bigg)dvdu=\frac{377e}{25920}+\frac{733}{1296}.
\end{multline*}
As explained above, this gives
$$
S(W_{\bullet,\bullet}^{{S}};{C})\leqslant\frac{32}{405}+\frac{3}{24}\int\limits_0^{\frac{3}{5}}\int\limits_0^\infty \mathrm{vol}\big(P(u)\big\vert_{{S}}-vC_{0})\big)dvdu=\frac{377e}{25920} + \frac{4177}{6480}.
$$
Since $e \in \{0,2,4,6,8,10\}$, we conclude that $S(W_{\bullet,\bullet}^{{S}};{C})<1$. Then $\beta(\mathbf{F})>0$ by \eqref{equation:Kento-curve}.
\end{proof}

Let $f\colon\widetilde{S}\to S$ be the~blow up of the~point $P$,
and let $F$ be the~$f$-exceptional curve. Write
$$
f^*\big(N(u)\big\vert_S\big)=\widetilde{d}(u)F+\widetilde{N}^\prime(u),
$$
where $\widetilde{N}^\prime(u)$ is the~strict transform of the~divisor $N(u)\vert_{S}$ on the~surface $\widetilde{S}$,
and $\widetilde{d}(u)=\mathrm{mult}_P(N(u)\vert_S)$.
For every $u\in [0,\tau]$, we set
$$
\widetilde{t}(u)=\sup\Big\{v\in \mathbb R_{\geqslant 0} \ \big|\ \text{the divisor $f^*\big(P(u)\big|_S\big)-vF$ is pseudo-effective}\Big\}.
$$
For $v\in [0, \widetilde{t}(u)]$, we let $\widetilde{P}(u,v)$ be the~positive part of the~Zariski decomposition of $f^*(P(u)|_S)-vF$,
and we let $\widetilde{N}(u,v)$ be its negative part. As above, we let
$$
S\big(W^S_{\bullet,\bullet};F\big)=\frac{3}{(-K_X)^3}\int\limits_0^{\tau}\widetilde{d}(u)\Big(P(u)\big\vert_{S}\Big)^2du+
\frac{3}{(-K_X)^3}\int\limits_0^\tau\int\limits_0^{\infty}\mathrm{vol}\big(f^*\big(P(u)\big|_S\big)-vF\big)dvdu,
$$
which we can rewrite as
$$
S\big(W^S_{\bullet,\bullet};F\big)=\frac{3}{(-K_X)^3}\int\limits_0^{\tau}\widetilde{d}(u)\big(P(u,0)\big)^2du+
\frac{3}{(-K_X)^3}\int\limits_0^\tau\int\limits_0^{\widetilde{t}(u)}\big(\widetilde{P}(u,v)\big)^2dvdu.
$$
For every point $O\in F$, we let
$$
F_O\big(W_{\bullet,\bullet,\bullet}^{\widetilde{S},F}\big)=\frac{6}{(-K_X)^3}
\int\limits_0^\tau\int\limits_0^{\widetilde{t}(u)}\big(\widetilde{P}(u,v)\cdot F\big)\cdot \mathrm{ord}_O\big(\widetilde{N}^\prime(u)\big|_F+\widetilde{N}(u,v)\big|_F\big)dvdu,
$$
and
$$
S\big(W_{\bullet, \bullet,\bullet}^{\widetilde{S},F};O\big)=\frac{3}{(-K_X)^3}\int\limits_0^\tau\int\limits_0^{\widetilde{t}(u)}\big(\widetilde{P}(u,v)\cdot F\big)^2dvdu+F_O\big(W_{\bullet,\bullet,\bullet}^{\widetilde{S},F}\big).
$$
Then it follows from \cite{AbbanZhuang,Book} that
\begin{equation}
\label{equation:Kento-point}
\frac{A_{X}(\mathbf{F})}{S_X(\mathbf{F})}\geqslant
\min\Bigg\{\frac{1}{S_X({S})},\frac{2}{S\big(W^{S}_{\bullet,\bullet};F\big)},\inf_{O\in F}\frac{1}{S\big(W_{\bullet, \bullet,\bullet}^{\widetilde{S},F};O\big)}\Bigg\}.
\end{equation}
In the~next two lemmas, we show how to apply this inequality to prove that $\beta(\mathbf{F})>0$ under certain generality conditions on the~position of the~point $P$.

\begin{lemma}
\label{lemma:dP4-blow-up}
Let $S$ be a~general surface in $|H_{\mathbb{P}^3}|$ such that $P\in S$.
Suppose $P\not\in E_{\mathbb{P}^3}$, $-K_S$ is ample, and $P$ is not contained in a $(-1)$-curve in $S$.
Then $\beta(\mathbf{F})>0$.
\end{lemma}

\begin{proof}
Observe that $\pi(S)$ is a general plane in $\mathbb{P}^3$ that contains $\pi(P)$,
and $\pi$ induces a birational morphism $\varpi\colon S\to\pi(S)$ that blows up the~points $\pi(S)\cap C_5$.
Let $\mathbf{e}_1$, $\mathbf{e}_2$, $\mathbf{e}_3$, $\mathbf{e}_4$, $\mathbf{e}_5$ be the~$\varpi$-exceptional curves.
Then $E_{\mathbb{P}^3}\big\vert_{S}=\mathbf{e}_1+\mathbf{e}_2+\mathbf{e}_3+\mathbf{e}_4+\mathbf{e}_5$.

Let $L=H_{\mathbb{P}^3}\vert_{S}$. For $i\in\{1,2,3,4,5\}$,
the~pencils $|L-\mathbf{e}_i|$ and $|2L+\mathbf{e}_i-\mathbf{e}_1-\mathbf{e}_2-\mathbf{e}_3-\mathbf{e}_4-\mathbf{e}_5|$~contain irreducible curves that pass through the~point $P$.
Denote these curves by $Z_i$ and $Z_i^\prime$, respectively.
Then $\varpi(Z_i)$ is the~line in $\pi(S)$ that passes through $\varpi(P)$ and $\varpi(\mathbf{e}_i)$,
and $\varpi(Z_i^\prime)$ is the~conic that passes through $\varphi(P)$ and all points among
$\varpi(\mathbf{e}_1)$, $\varpi(\mathbf{e}_2)$, $\varpi(\mathbf{e}_3)$, $\varpi(\mathbf{e}_4)$, $\varpi(\mathbf{e}_5)$ except for $\varpi(\mathbf{e}_i)$.
Set
\begin{align*}
Z&=\sum_{i=1}^{5}Z_i\sim 5L-(\mathbf{e}_1+\mathbf{e}_2+\mathbf{e}_3+\mathbf{e}_4+\mathbf{e}_5),\\
Z^\prime&=\sum_{i=1}^{5}Z_i^\prime\sim 10L-4(\mathbf{e}_1+\mathbf{e}_2+\mathbf{e}_3+\mathbf{e}_4+\mathbf{e}_5).
\end{align*}
Let $\widetilde{Z}$ and $\widetilde{Z}^\prime$ be the~proper transforms on $\widetilde{S}$ of the~curves $Z$ and $Z^\prime$, respectively.
On the~surface $\widetilde{S}$, we have $F^2=-1$, $\widetilde{Z}\cdot\widetilde{Z}^\prime=F\cdot\widetilde{Z}=F\cdot\widetilde{Z}^\prime=5$, $\widetilde{Z}^{2}=(\widetilde{Z}^\prime)^{2}=-5$.
Using Example~\ref{example:AZ-surfaces-HP3}, we get $\tau=\frac{3}{2}$ and
$$
f^{*}\big(P(u)\big\vert_{S}\big)-vF \sim_{\mathbb{R}}\left\{\aligned
&\frac{3-2u}{5}\widetilde{Z}+\frac{1+u}{10}\widetilde{Z}^\prime+\frac{7-3u-2v}{2}F  \ \text{ for } 0\leqslant u\leqslant 1, \\
&\frac{3-2u}{5}\big(\widetilde{Z}+\widetilde{Z}^\prime\big)+(6-4u-v)F\ \text{ for } 1\leqslant u\leqslant \frac{3}{2}.
\endaligned
\right.
$$
This gives
$$
\widetilde{t}(u)=\left\{\aligned
&\frac{7-3u}{2}\ \text{ for } 0\leqslant u\leqslant 1, \\
&6-4u\ \text{ for } 1\leqslant u\leqslant \frac{3}{2}.
\endaligned
\right.
$$
Furthermore, if $0\leqslant u\leqslant 1$, then
$$
\widetilde{P}(u,v)=
\left\{\aligned
&\frac{3-2u}{5}\widetilde{Z}+\frac{1+u}{10}\widetilde{Z}^\prime+\frac{7-3u-2v}{2}F  \ \text{ for } 0\leqslant v\leqslant 3-u, \\
&\frac{18-7u-5v}{5}\widetilde{Z}+\frac{1+u}{10}\widetilde{Z}^\prime+\frac{7-3u-2v}{2}F \ \text{ for } 3-u \leqslant v \leqslant \frac{7-3u}{2},
\endaligned
\right.
$$
and
$$
\widetilde{N}(u,v)=
\left\{\aligned
&0\ \text{ for } 0\leqslant v\leqslant 3-u, \\
&(v+u-3)\widetilde{Z}\ \text{ for } 3-u\leqslant v\leqslant\frac{7-3u}{2},
\endaligned
\right.
$$
which gives
$$
\big(\widetilde{P}(u,v)\big)^2=
\left\{\aligned
&u^2-v^2-8u+11\ \text{ for } 0\leqslant v\leqslant 3-u, \\
&2(4-u-v)(7-3u-2v)\ \text{ for } 3-u\leqslant v\leqslant \frac{7-3u}{2},
\endaligned
\right.
$$
and
$$
\widetilde{P}(u,v)\cdot F=
\left\{\aligned
&v\ \text{ for } 0\leqslant v\leqslant 3-u, \\
&15-5u-4v\ \text{ for } 3-u\leqslant v\leqslant\frac{7-3u}{2}.
\endaligned
\right.
$$
If $1\leqslant u\leqslant \frac{3}{2}$, then
$\widetilde{P}(u,v)=\frac{3-2u}{5}(\widetilde{Z}+\widetilde{Z})+(6-4u-v)F$
and $\widehat{N}(u,v)=0$ for $v\in[0,6-4u]$, so that
$$
\widehat{P}(u,v)^2=(6-4u-v)(6-4u+v)
$$
and $\widehat{P}(u,v)\cdot F=v$ for every $v\in[0,6-4u]$.

Set $R=E_Q\vert_{S}$. Then $R$ is smooth curve, since $S$ is general surface in $|H_{\mathbb{P}^3}|$ that passes through~$P$.
Let $\widetilde{R}$ be the~proper transform of the~curve $R$ on the~surface $\widetilde{R}$.
Then it follows from Example~\ref{example:AZ-surfaces-HP3} that
$$
\widetilde{N}^\prime(u)=\left\{\aligned
&0  \ \text{ for } 0\leqslant u\leqslant 1, \\
&(u-1)\widetilde{R} \ \text{ for } 1\leqslant u\leqslant \frac{3}{2}.
\endaligned
\right.
$$
If $0\leqslant u\leqslant 1$, we have $\widetilde{d}(u)=0$.
Similarly, if $1\leqslant u\leqslant \frac{3}{2}$ and $R$ does not contain $P$, then $\widetilde{d}(u)=0$.
Finally, if $1\leqslant u\leqslant \frac{3}{2}$ and $P\in R$, then $\widetilde{d}(u)=(u-1)$.

Using the~data collected above, we can compute $S(W^S_{\bullet,\bullet};F)$.
Namely, if $P\in E_{Q}$, then
\begin{multline*}
S\big(W^S_{\bullet,\bullet};F\big)=\frac{1}{8}\int\limits_1^{\frac{3}{2}}(u-1)(16u^2-48u+36)du+
\frac{1}{8}\int\limits_0^{\frac{3}{2}}\int\limits_0^{\widetilde{t}(u)}\big(\widetilde{P}(u,v)\big)^2dvdu=\\
=\frac{1}{96}+\frac{1}{8}\int\limits_{0}^{1}\int\limits_{0}^{3-u}u^2-v^2-8u+11dvdu+\frac{1}{8}\int\limits_{0}^{1}\int\limits_{3-u}^{\frac{7-3u}{2}}2(4-u-v)(7-3u-2v)dvdu+\\
+\frac{1}{8}\int\limits_{1}^{\frac{3}{2}}\int\limits_{0}^{6-4u}(6-4u-v)(6-4u+v)dvdu=\frac{1}{96}+\frac{655}{384}+\frac{1}{12}=\frac{691}{384}<2.
\end{multline*}
Similarly, if $P\not\in E_{Q}$, then $S(W^S_{\bullet,\bullet};F)=\frac{655}{384}+\frac{1}{12}=\frac{229}{128}<2$.

Now, let $O$ be any point in $F$. Then
\begin{multline*}
\quad\quad\quad S\big(W_{\bullet, \bullet,\bullet}^{\widetilde{S},F};O\big)=\frac{1}{8}\int\limits_0^{\frac{3}{2}}\int\limits_0^{\widetilde{t}(u)}\big(\widetilde{P}(u,v)\cdot F\big)^2dvdu+F_O\big(W_{\bullet,\bullet,\bullet}^{\widetilde{S},F}\big)=\frac{155}{192}+F_O\big(W_{\bullet,\bullet,\bullet}^{\widetilde{S},F}\big)=\\
=\frac{1}{8}\int\limits_{0}^{1}\int\limits_{0}^{3-u}v^{2}dvdu+\frac{1}{8}\int\limits_{0}^{1}\int\limits_{3-u}^{\frac{7-3u}{2}}(15-5u-4v)^2dvdu+\frac{1}{8}\int\limits_{1}^{\frac{3}{2}}\int\limits_{0}^{6-4u}v^2dvdu+F_O\big(W_{\bullet,\bullet,\bullet}^{\widetilde{S},F}\big)=\frac{163}{192}+F_O\big(W_{\bullet,\bullet,\bullet}^{\widetilde{S},F}\big).
\end{multline*}
Moreover, if $O\in\widetilde{R}\cap\widetilde{Z}$, then we compute $F_O(W_{\bullet,\bullet,\bullet}^{\widetilde{S},F})$ as follows:
\begin{multline*}
F_O\big(W_{\bullet,\bullet,\bullet}^{\widetilde{S},F}\big)=
\frac{1}{4}\int\limits_0^{\frac{3}{2}}\int\limits_0^{\widetilde{t}(u)}\big(\widetilde{P}(u,v)\cdot F\big)\cdot \mathrm{ord}_O\big(\widetilde{N}^\prime(u)\big|_F\big)dvdu+
\frac{1}{4}\int\limits_0^{\frac{3}{2}}\int\limits_0^{\widetilde{t}(u)}\big(\widetilde{P}(u,v)\cdot F\big)\cdot \mathrm{ord}_O\big(\widetilde{N}(u,v)\big|_F\big)dvdu=\\
=\frac{1}{4}\int\limits_1^{\frac{3}{2}}\int\limits_0^{6-4u}\big(\widetilde{P}(u,v)\cdot F\big)(u-1)\big(\widetilde{R}\cdot F\big)_Odvdu+
\frac{1}{4}\int\limits_0^{1}\int\limits_{3-u}^{\frac{7-3u}{2}}\big(\widetilde{P}(u,v)\cdot F\big)(v+u-3)\big(\widetilde{Z}\cdot F\big)_{O}dvdu=\\
=\frac{1}{4}\int\limits_{1}^{\frac{3}{2}}\int\limits_{0}^{6-4u}v(u-1)dvdu+\frac{1}{4}\int\limits_{0}^{1}\int\limits_{3-u}^{\frac{7-3u}{2}}(15-5u-4v)(v+u-3)dvdu=\frac{1}{96}+\frac{7}{384}=\frac{11}{384},\quad\quad\quad\quad
\end{multline*}
because the~curve $\widetilde{R}$ intersects $F$ transversally, and every irreducible component of the~curve $\widetilde{Z}$ also intersects $F$ transversally.
Hence, if $O\in\widetilde{R}\cap\widetilde{Z}$, then $S(W_{\bullet, \bullet,\bullet}^{\widetilde{S},F};O)=\frac{337}{384}<1$.
Similar computations imply that $S(W_{\bullet, \bullet,\bullet}^{\widetilde{S},F};O)<\frac{337}{384}<1$ if $O\not\in\widetilde{R}$ or $O\not\in\widetilde{Z}$.
Thus, using \eqref{equation:Kento-point}, we see that $\beta(\mathbf{F})>0$.
\end{proof}

\begin{lemma}
\label{lemma:dP4-blow-up-conic}
Let $S$ be a~general surface in $|H_{\mathbb{P}^3}|$ such that $P\in S$.
Suppose $P\not\in E_{\mathbb{P}^3}$, $-K_S$ is ample, and $P$ is contained in a $(-1)$-curve $B\subset S$
such that $\pi(B)$ is a conic. Then $\beta(\mathbf{F})>0$.
\end{lemma}

\begin{proof}
Let us use notations introduced in the~proof of Lemma~\ref{lemma:dP4-blow-up},
and let $\widetilde{B}$ be the~proper transform on the~surface $\widetilde{S}$ of the~curve $B$.
Observe that $\widetilde{B}$ and $\widetilde{Z}$ are disjoint, and $\widetilde{B}^2=-2$ on the~surface~$\widetilde{S}$.
Moreover, it follows from Example~\ref{example:AZ-surfaces-HP3} that $\tau=\frac{3}{2}$ and
$$
f^{*}\big(P(u)\big\vert_{S}\big)-vF \sim_{\mathbb{R}}\left\{\aligned
&\frac{2-u}{3}\widetilde{Z}+\frac{1+u}{3}\widetilde{B}+\frac{11-4u-3v}{3}F  \ \text{ for } 0\leqslant u\leqslant 1, \\
&\frac{3-2u}{3}\big(\widetilde{Z}+2\widetilde{B}\big)+\frac{21-14u-3v}{3}F\ \text{ for } 1\leqslant u\leqslant \frac{3}{2}.
\endaligned
\right.
$$
This gives
$$
\widetilde{t}(u)=\left\{\aligned
&\frac{11-4u}{3}\ \text{ for } 0\leqslant u\leqslant 1, \\
&\frac{21-14u}{3}\ \text{ for } 1\leqslant u\leqslant \frac{3}{2}.
\endaligned
\right.
$$
Furthermore, if $0\leqslant u\leqslant 1$, then
$$
\widetilde{P}(u,v)=
\left\{\aligned
&\frac{2-u}{3}\widetilde{Z}+\frac{1+u}{3}\widetilde{B}+\frac{11-4u-3v}{3}F  \ \text{ for } 0\leqslant v\leqslant 3-2u, \\
&\frac{2-u}{3}\widetilde{Z}+\frac{11-4u-3v}{6}\big(\widetilde{B}+2F\big)  \ \text{ for } 3-2u \leqslant v \leqslant 3-u,\\
&\frac{11-4u-3v}{6}\big(2L+\widetilde{B}+2F\big)   \ \text{ for } 3-u \leqslant v \leqslant \frac{11-4u}{3},
\endaligned
\right.
$$
and
$$
\widetilde{N}(u,v)=
\left\{\aligned
&0  \ \text{ for } 0\leqslant v\leqslant 3-2u, \\
&\frac{v+2u-3}{2}\widetilde{B}  \ \text{ for } 3-2u \leqslant v \leqslant 3-u,\\
&\frac{v+2u-3}{2}\widetilde{B}+(v+u-3)\widetilde{Z}  \ \text{ for } 3-u \leqslant v \leqslant \frac{11-4u}{3},
\endaligned
\right.
$$
which gives
$$
\big(\widetilde{P}(u,v)\big)^2=
\left\{\aligned
&u^2-v^2-8u+11  \ \text{ for } 0\leqslant v\leqslant 3-2u, \\
&\frac{31}{2}-14u-3v+3u^2-\frac{v^2}{2}+2vu  \ \text{ for } 3-2u \leqslant v \leqslant 3-u,\\
&\frac{(11-4u-3v)^2}{2}\ \text{ for } 3-u \leqslant v \leqslant \frac{11-4u}{3},
\endaligned
\right.
$$
and
$$
\widetilde{P}(u,v)\cdot F=
\left\{\aligned
&v  \ \text{ for } 0\leqslant v\leqslant 3-2u, \\
&\frac{3-2u+v}{2}  \ \text{ for } 3-2u \leqslant v \leqslant 3-u,\\
&\frac{33-12u-9v}{2}   \ \text{ for } 3-u \leqslant v \leqslant \frac{11-4u}{3}.
\endaligned
\right.
$$
Similarly, if $1\leqslant u\leqslant \frac{3}{2}$, then
$$
\widetilde{P}(u,v)=
\left\{\aligned
&\frac{3-2u}{3}\big(\widetilde{Z}+2\widetilde{B}\big)+\frac{21-14u-3v}{3}F \ \text{ for } 0\leqslant v\leqslant 3-2u, \\
&\frac{3-2u}{3}\widetilde{Z}+\frac{21-14u-3v}{6}\big(\widetilde{B}+2F\big) \ \text{ for } 3-2u \leqslant v \leqslant 6-4u,\\
&\frac{21-14u-3v}{6}\big(2\widetilde{Z}+\widetilde{B}+2F\big)     \ \text{ for } 6-4u \leqslant v \leqslant \frac{21-14u}{3},
\endaligned
\right.
$$
and
$$
\widetilde{N}(u,v)=
\left\{\aligned
&0  \ \text{ for } 0\leqslant v\leqslant 3-2u, \\
&\frac{v+2u-3}{2}\widetilde{B} \ \text{ for } 3-2u \leqslant v \leqslant 6-4u,\\
&\frac{v+2u-3}{2}\widetilde{B}+(v+4u-6)\widetilde{Z} \ \text{ for } 6-4u \leqslant v \leqslant \frac{21-14u}{3},
\endaligned
\right.
$$
which gives
$$
\big(\widetilde{P}(u,v)\big)^2=
\left\{\aligned
&(6-4u-v)(6-4u+v) \ \text{ for } 0\leqslant v\leqslant 3-2u, \\
&\frac{81}{2}-54u-3v+18u^2-\frac{v^2}{2}+2vu  \ \text{ for } 3-2u \leqslant v \leqslant 6-4u,\\
&\frac{(21-14u-3v)^2}{2}  \ \text{ for } 6-4u \leqslant v \leqslant \frac{21-14u}{3},
\endaligned
\right.
$$
and
$$
\widetilde{P}(u,v)\cdot F=
\left\{\aligned
&v  \ \text{ for } 0\leqslant v\leqslant 3-2u, \\
&\frac{3-2u+v}{2}  \ \text{ for } 3-2u \leqslant v \leqslant 6-4u,\\
&\frac{63-42u-9v}{2}  \ \text{ for } 6-4u \leqslant v \leqslant \frac{21-14u}{3}.
\endaligned
\right.
$$
Thus, as in the~proof of Lemma~\ref{lemma:dP4-blow-up}, we compute
$$
S\big(W^S_{\bullet,\bullet};F\big)=
\left\{\aligned
&\frac{523}{288}  \ \text{ if } P\in E_{Q}, \\
&\frac{65}{36}  \ \text{ if } P\not\in E_{Q},
\endaligned
\right.
$$
so that $S(W^S_{\bullet,\bullet};F)<2$. Similarly, if $O$ is a point in $F$, then
$$
S\big(W_{\bullet, \bullet,\bullet}^{\widetilde{S},F};O\big)=
\left\{\aligned
&\frac{257}{288}   \ \text{ if } O\in \widetilde{B}\cap\widetilde{R}, \\
&\frac{119}{144}   \ \text{ if } O\in \widetilde{Z}\cap\widetilde{R},\\
&\frac{127}{144}   \ \text{ if } O\in\widetilde{B}\ \text{and}\ O\not\in\widetilde{R}, \\
&\frac{235}{288}   \ \text{ if } O\in\widetilde{Z}\ \text{and}\ O\not\in\widetilde{R},\\
&\frac{307}{384}   \ \text{ if } O\not\in\widetilde{B}\cup\widetilde{Z}\ \text{and}\ O\in\widetilde{R}\\
&\frac{101}{128}   \ \text{ if } O\not\in\widetilde{B}\cup\widetilde{Z}\cup\widetilde{R}.
\endaligned
\right.
$$
Therefore, using \eqref{equation:Kento-point}, we see that $\beta(\mathbf{F})>0$.
\end{proof}

On the other hand, we have the following purely geometric result.

\begin{lemma}
\label{lemma:SP3-dP4}
Suppose that $P\not\in E_{\mathbb{P}^3}$.
Let $S$ be a general surface in $|H_{\mathbb{P}^3}|$ such that $S$ passes through~$P$.
Then $-K_{S}$ is ample.
Further, if $P$ is contained in a $(-1)$-curve $B\subset S$,
then $\pi(B)$ is a smooth conic.
\end{lemma}

\begin{proof}
The surface $\pi(S)$ is a general plane in $\mathbb{P}^3$ that contains the~point $\pi(P)$. Write
$$
\pi(S)\cap C_5=P_1\cup P_2\cup P_3\cup P_4\cup P_5,
$$
where $P_1$, $P_2$, $P_3$, $P_4$, $P_5$ are distinct points.
Then $\pi$ induces a birational morphism~\mbox{$\varpi\colon S\to \pi(S)$}, which is a blow up of the~intersection points $P_1$, $P_2$, $P_3$, $P_4$, $P_5$.
Thus, to prove that $-K_S$ is ample, we must show that at most two points among these five are contained in a line.

If three points among $P_1$, $P_2$, $P_3$, $P_4$, $P_5$ are contained in a line $\ell$,
it is a trisecant of the curve~$C_5$, the line $\ell$ is contained in $\pi(E_Q)$, and its proper transform on $X$
is a fiber of the projection $E_Q\to C_5^\prime$.
However, the planes containing $\pi(P)$ and a trisecant of the curve $C_5$ form a one-dimensional family.
Hence, a general plane in $\mathbb{P}^3$ that passes through $\pi(P)$ does not contain trisecants of the curve $C_5$, so that
at most two points among $P_1$, $P_2$, $P_3$, $P_4$, $P_5$ are contained in a line.
Thus, $-K_S$ is ample.

Now, we suppose that $P$ is contained in a $(-1)$-curve $B\subset S$.
If $\pi(B)$ is not a conic, it must be a secant of the curve $C_5$ that contains $\pi(P)$.
Let $\phi\colon\mathbb{P}^3\dasharrow\mathbb{P}^2$ be the linear projection from~$\pi(P)$.
Since $\pi(P)\not\in C_5$, $\phi$ induces a birational morphism $C_5\to\phi(C_5)$, and
$\phi(C_5)$ is a singular irreducible curve of degree $5$.
Moreover, if $\ell$ is a secant of the curve $C_5$ that contains $\pi(P)$, then $\phi(\ell)$ is a singular point of the curve $\phi(C_5)$.
Since this curve has finitely many singular points, we conclude that there are finitely many secants of the curve $C_5$ that passes through $\pi(P)$.
This shows that $\pi(S)$ does not contain secants of the curve $C_5$ that pass through $\pi(P)$,
because $\pi(S)$ is a general plane in $\mathbb{P}^3$ that contains the~point~$\pi(P)$.
So, we conclude that $\pi(B)$ must be a conic.
\end{proof}

Hence, applying Lemmas~\ref{lemma:E}, \ref{lemma:dP4-blow-up}, \ref{lemma:dP4-blow-up-conic}, \ref{lemma:SP3-dP4}, we obtain

\begin{corollary}
\label{corollary:Z-EP3}
If $\beta(\mathbf{F})\leqslant 0$, then $Z$ is a fiber of the~projection $E_{\mathbb{P}^3}\to C_5$.
\end{corollary}

\begin{remark}
\label{remark:Z-EP3-1-3}
By \cite[Corollary~4.14]{Zhuang}, Corollary~\ref{corollary:Z-EP3} implies both Corollaries~\ref{corollary:1} and  Corollaries~\ref{corollary:3}.
\end{remark}

To complete the proof of the Main Theorem, we may assume $Z$ is a fiber of the~projection $E_{\mathbb{P}^3}\to C_5$.

Note that $Z\not\subset E_Q$, since $E_Q$ is irrational.
Thus, we can choose  $P\in Z$ such that $P\not\in E_Q$ either.
Now, let $S$ be a general surface in $|H_{Q}|$ such that $P\in S$.
Then the surface $q(S)$ is a general~hyperplane section of the~quadric $Q$ that contains $q(P)$,
so $q(S)\cong\mathbb{P}^1\times\mathbb{P}^1$ and
$q(S)\cap C_5^\prime=P_1\cup P_2\cup P_3\cup P_4\cup P_5$,
where $P_1$, $P_2$, $P_3$, $P_4$, $P_5$ are distinct points in $C_5^\prime$.
Then the~morphism $q\colon X\to Q$ induces a birational
morphism $S\to q(S)$ that blows up the~points $P_1$, $P_2$, $P_3$, $P_4$, $P_5$.

\begin{lemma}
\label{lemma:SP3-dP3}
The divisor $-K_{S}$ is ample.
\end{lemma}

\begin{proof}
Observe that $q(S)$ is a general hyperplane section of the quadric $Q$ that passes through $q(P)$,
and it follows from the adjunction formula that $-K_S\sim H_{\mathbb{P}^3}\big\vert_{S}$.
Thus, if $-K_S$ is not ample, $S$ contains a fiber $\ell$ of the~projection $E_{\mathbb{P}^3}\to C_5$
such that $q(\ell)$ is a secant line of the curve $C_5^\prime$.
On the other hand, hyperplane sections of $Q$ that contain $q(P)$ and a secant line of the curve $C_5^\prime$
form a two-dimensional family.
So, we may assume that $q(S)$ is not one of them, which implies that $-K_S$ is ample.
\end{proof}

Since $-K_{S}$ is ample, the~morphism $\pi$ induces an isomorphism $S\cong\pi(S)$, and $\pi(S)$ is a smooth cubic surface in $\mathbb{P}^3$
that contains the~curve $C_5$.
Let us identify $S$ with the~smooth cubic surface $\pi(S)$.
Using this identification, we see that $C_5=E_{\mathbb{P}^3}\cap S$.
Then $P\in C_5$, since $P\in E_{\mathbb{P}^3}$.

Let $L_1$ and $L_2$ be the~proper transforms on $S$ of two rulings of the~surface $q(S)\cong\mathbb{P}^1\times\mathbb{P}^1$
that pass through the~point $q(P)$. Then $L_1$ and $L_2$ are conics in $S$, because $q(L_1)$ and $q(L_2)$ do not contain
any of the~points $P_1$, $P_2$, $P_3$, $P_4$, $P_5$, since we assume that $q(S)$ is a general hyperplane section of the~quadric $Q$ that contains the~point $q(P)$.
Moreover, it follows from Example~\ref{example:AZ-surfaces-HPQ} that
$$
P(u)\big\vert_{S}\sim_{\mathbb{R}} \left\{\aligned
& -K_S+(1-u)(L_1+L_2) \ \text{ for } 0\leqslant u\leqslant 1, \\
& (4-3u)(-K_S)\ \text{ for } 1\leqslant u\leqslant \frac{4}{3},\\
\endaligned
\right.
$$
and
$$
N(u)\big\vert_{S}=\left\{\aligned
&0\ \text{ for } 0\leqslant u\leqslant 1, \\
&(u-1)C_5\ \text{ for } 1\leqslant u\leqslant \frac{4}{3},\\
\endaligned
\right.
$$
Let $T_P$ be the~unique curve in the~linear system $|-K_S|$ that is singular at $P$.
Then $T_P$ is cut out by the~hyperplane in $\mathbb{P}^3$ that is tangent to $S$ at the~point $P$.
In particular, the~curve $T_P$ is reduced.

\begin{lemma}
\label{lemma:Tp}
We have the~following five possible cases:
\begin{itemize}
\item $T_P$ is an irreducible cubic curve,
\item $T_P=\ell+C_2$, where $\ell$ is a line, $C_2$ is a smooth conic such that $C_2\ne L_1$ and $C_2\ne L_2$,
\item $T_P=\ell_1+\ell_2+\ell_3$, where $\ell_1$, $\ell_2$, $\ell_3$ are lines such that $P=\ell_1\cap\ell_2$ and $P\not\in\ell_3$.
\end{itemize}
\end{lemma}

\begin{proof}
A priori, since $T_P$ is a reduced cubic curve, we may have the~following cases:
\begin{enumerate}
\item $T_P$ is an irreducible curve,
\item $T_P=\ell+L_1$, where $\ell$ is a line,
\item $T_P=\ell+L_2$, where $\ell$ is a line,
\item $T_P=\ell+C_2$, where $\ell$ is a line, $C_2$ is a smooth conic such that $C_2\ne L_1$ and $C_2\ne L_2$,
\item $T_P=\ell_1+\ell_2+\ell_3$, where $\ell_1$, $\ell_2$, $\ell_3$ are lines such that $P=\ell_1\cap\ell_2$ and $P\not\in\ell_3$,
\item $T_P=\ell_1+\ell_2+\ell_3$, where $\ell_1$, $\ell_2$, $\ell_3$ are lines such that $P=\ell_1\cap\ell_2\cap\ell_3$.
\end{enumerate}
If $T_P=\ell_1+\ell_2+\ell_3$ for three lines $\ell_1$, $\ell_2$, $\ell_3$ such that $P=\ell_1\cap\ell_2\cap\ell_3$, then
$$
2=-K_S\cdot L_1=\big(\ell_1+\ell_2+\ell_3\big)\cdot L_1\geqslant \sum_{i=1}^3\big(\ell_1\cdot L_1\big)_P\geqslant 3,
$$
which is absurd. Thus, we see that the~last case is impossible. To complete the~proof, we must show that the~second and the~third cases are also impossible.

Suppose that $T_P=\ell+L_1$ for some line $\ell$.
Then $q(\ell)$ is a twisted cubic curve in $Q$ that contains all intersection points $P_1$, $P_2$, $P_3$, $P_4$, $P_5$.
In particular, we see that $E_Q\cap \ell\geqslant 5$.
On the other hand, the~curve $q(\ell)$ is not contained in the~surface $q(E_{\mathbb{P}^3})$, because the~only rational curves in the ruled irrational surface $E_{\mathbb{P}^3}$
are fibers of the~natural projection $E_{\mathbb{P}^3}\to C_5$, which are mapped to lines by $q$.
Therefore, since $P\in E_{\mathbb{P}^3}$ by assumption, we have
$$
1\leqslant\big(E_{\mathbb{P}^3}\cdot\ell\big)_{P}\leqslant E_{\mathbb{P}^3}\cdot\ell=\big(5H_Q-3E_Q\big)\cdot\ell=15-3E_Q\cdot\ell\leqslant 0,
$$
which is absurd. This shows that the~conic $L_1$ cannot be an irreducible component of the curve $T_P$.
Similarly, we see $L_2$ is also not an irreducible component of the curve $T_P$.
\end{proof}

From Example~\ref{example:AZ-surfaces-HPQ}, we know that $S_X(S)<1$.
So, it follows from \eqref{equation:Kento-point} that $\beta(\mathbf{F})\leqslant 0$ if $S(W^S_{\bullet,\bullet};F)<2$,
and $S(W_{\bullet, \bullet,\bullet}^{\widetilde{S},F};O)<1$ for every point $O\in F$.
Let us check these conditions.

Let $\widetilde{L}_1$, $\widetilde{L}_2$, $\widetilde{T}_P$ be proper transforms on $\widetilde{S}$ of the~curves $L_1$, $L_2$, $T_P$.
Then
\begin{equation}
\label{equation:cubic-surface-blown-up}
f^{*}\big(P(u)\big\vert_{S}\big)-vF \sim_{\mathbb{R}}\left\{\aligned
&\widetilde{T}_P+(1-u)\big(\widetilde{L}_1+\widetilde{L}_2\big)+(4-2u-v)F \ \text{ for } 0\leqslant u\leqslant 1, \\
&(4-3u)\widetilde{T}_P+(8-6u-v)F\ \text{ for } 1\leqslant u\leqslant \frac{4}{3}.
\endaligned
\right.
\end{equation}
Hence, if $0\leqslant u\leqslant 1$, then $\widetilde{t}(u)=4-2u$.
Similarly, if $1\leqslant u\leqslant \frac{4}{3}$, then $\widetilde{t}(u)=8-6u$.
Moreover, since the curve $C_5=E_{\mathbb{P}^3}\cap S$ is smooth, we have
$\widetilde{d}(u)=0$
for $0\leqslant u\leqslant 1$,
and $\widetilde{d}(u)=u-1$ for $1\leqslant u\leqslant \frac{4}{3}$.
Finally, let $\widetilde{C}_5$ be the proper transform on $\widetilde{S}$ of the curve $C_5$. Then
$$
\widetilde{N}^\prime(u)=\left\{\aligned
&0  \ \text{ for } 0\leqslant u\leqslant 1, \\
&(u-1)\widetilde{C}_5 \ \text{ for } 1\leqslant u\leqslant \frac{4}{3}.
\endaligned
\right.
$$
Now, we can compute $S(W^S_{\bullet,\bullet};F)$ and $S(W_{\bullet, \bullet,\bullet}^{\widetilde{S},F};O)$ for every point $O\in F$.

\begin{lemma}
\label{lemma:cubic-surface-case-a}
If $T_P$ is irreducible, then $S(W^S_{\bullet,\bullet};F)<2$ and $S(W_{\bullet, \bullet,\bullet}^{\widetilde{S},F};O)<1$ for every point $O\in F$.
\end{lemma}

\begin{proof}
Suppose that the curve $T_P$ is irreducible. Then $\widetilde{S}$ is a smooth del Pezzo surface of degree $2$.
Note that $\widetilde{L}_1$, $\widetilde{L}_1$, $\widetilde{T}_P$ are disjoint $(-1)$-curves on $\widetilde{S}$.
If $0\leqslant u\leqslant\frac{1}{2}$, it follows from \eqref{equation:cubic-surface-blown-up} that
$$
\widetilde{P}(u,v)=
\left\{\aligned
&\widetilde{T}_P+(1-u)\big(\widetilde{L}_1+\widetilde{L}_2\big)+(4-2u-v)F\ \text{for}\ 0\leqslant v\leqslant 3-u,\\
&\widetilde{T}_P+(4-2u-v)\big(\widetilde{L}_1+\widetilde{L}_2+F\big)\ \text{for}\ 3-u\leqslant v\leqslant \frac{7-4u}{2},\\
&(4-2u-v)\big(2\widetilde{T}_P+\widetilde{L}_1+\widetilde{L}_2+F\big)\ \text{for}\ \frac{7-4u}{2}\leqslant v\leqslant4-2u,
\endaligned
\right.
$$
and
$$
\widetilde{N}(u,v)=
\left\{\aligned
&0\ \text{for}\ 0\leqslant v\leqslant3-u,\\
&(v+u-3)\big(\widetilde{L}_1+\widetilde{L}_2\big)\ \text{for}\ 3-u\leqslant v\leqslant\frac{7-4u}{2},\\
&(v+u-3)\big(\widetilde{L}_1+\widetilde{L}_2\big)+(2v+4u-7)\widetilde{T}_P\ \text{for}\ \frac{7-4u}{2}\leqslant v\leqslant 4-2u,
\endaligned
\right.
$$
which gives
$$
\big(\widetilde{P}(u,v)\big)^2=
\left\{\aligned
&2u^2-v^2-12u+13\ \text{for}\ 0\leqslant v\leqslant3-u,\\
&4u^2+4uv+v^2-24u-12v+31\ \text{for}\ 3-u\leqslant v\leqslant \frac{7-4u}{2},\\
&5(2u+v-4)^2\ \text{for}\ \frac{7-4u}{2}\leqslant v\leqslant 4-2u,
\endaligned
\right.
$$
and
$$
\widetilde{P}(u,v)\cdot F=
\left\{\aligned
&v\ \text{for}\ 0\leqslant v\leqslant3-u,\\
&6-2u-v\ \text{for}\ 3-u\leqslant v\leqslant \frac{7-4u}{2},\\
&20-10u-5v\ \text{for}\ \frac{7-4u}{2}\leqslant v\leqslant 4-2u.
\endaligned
\right.
$$
Likewise, if $\frac{1}{2}\leqslant u\leqslant 1$, then it follows from \eqref{equation:cubic-surface-blown-up} that
$$
\widetilde{P}(u,v)=
\left\{\aligned
&\widetilde{T}_P+(1-u)\big(\widetilde{L}_1+\widetilde{L}_2\big)+(4-2u-v)F\ \text{for}\ 0\leqslant v\leqslant\frac{7-4u}{2},\\
&(4-2u-v)\big(2\widetilde{T}_P+F\big)+(1-u)\big(\widetilde{L}_1+\widetilde{L}_2\big)\ \text{for}\ \frac{7-4u}{2}\leqslant v\leqslant 3-u,\\
&(4-2u-v)\big(2\widetilde{T}_P+\widetilde{L}_1+\widetilde{L}_2+F\big)\ \text{for}\ 3-u\leqslant v\leqslant 4-2u,
\endaligned
\right.
$$
and
$$
\widetilde{N}(u,v)=
\left\{\aligned
&0\ \text{for}\ 0\leqslant v\leqslant\frac{7-4u}{2},\\
&(2v+4u-7)\widetilde{T}_P\ \text{for}\ \frac{7-4u}{2}\leqslant v\leqslant3-u,\\
&(v+u-3)\big(\widetilde{L}_1)+\widetilde{L}_2\big)+(2v+4u-7)\widetilde{T}_P\ \text{for}\ 3-u\leqslant v\leqslant 4-2u,
\endaligned
\right.
$$
which gives
$$
\big(\widetilde{P}(u,v)\big)^2=
\left\{\aligned
&2u^2-v^2-12u+13\ \text{for}\ 0\leqslant v\leqslant\frac{7-4u}{2},\\
&18u^2+16uv+3v^2-68u-28v+62\ \text{for}\ \frac{7-4u}{2}\leqslant v\leqslant 3-u,\\
&5(2u+v-4)^2\ \text{for}\ 3-u\leqslant v\leqslant 4-2u,
\endaligned
\right.
$$
and
$$
\widetilde{P}(u,v)\cdot F=
\left\{\aligned
&v\ \text{for}\ 0\leqslant v\leqslant\frac{7-4u}{2},\\
&14-8u-3v\ \text{for}\ \frac{7-4u}{2}\leqslant v\leqslant 3-u,\\
&20-10u-5v\ \text{for}\ 3-u\leqslant v\leqslant 4-2u.
\endaligned
\right.
$$
Similarly, if $1\leqslant u\leqslant\frac{4}{3}$, then
$$
\widetilde{P}(u,v)=
\left\{\aligned
&(4-3u)\widetilde{T}_P+(8-6u-v)F\ \text{for}\ 0\leqslant v\leqslant\frac{12-9u}{2},\\
&(8-6u-v)(2\widetilde{T}_P+F)\ \text{for}\ \frac{12-9u}{2}\leqslant v\leqslant 8-6u,
\endaligned
\right.
$$
and
$$
\widetilde{N}(u,v)=
\left\{\aligned
&0\ \text{for}\ 0\leqslant v\leqslant\frac{12-9u}{2},\\
&(2v+9u-12)\widetilde{T}_P\ \text{for}\ \frac{12-9u}{2}\leqslant v\leqslant 8-6u,
\endaligned
\right.
$$
which gives
$$
\big(\widetilde{P}(u,v)\big)^2=
\left\{\aligned
&27u^2-v^2-72u+48\ \text{for}\ 0\leqslant v\leqslant\frac{12-9u}{2},\\
&3(6u+v-8)^2\ \text{for}\ \frac{12-9u}{2}\leqslant v\leqslant 8-6u,
\endaligned
\right.
$$
and
$$
\widetilde{P}(u,v)\cdot F=
\left\{\aligned
&v\ \text{for}\ 0\leqslant v\leqslant\frac{12-9u}{2},\\
&24-18u-3v\ \text{for}\ \frac{12-9u}{2}\leqslant v\leqslant 8-6u.
\endaligned
\right.
$$
As in the~proof of Lemma~\ref{lemma:dP4-blow-up}, we get $S(W^S_{\bullet,\bullet};F)=\frac{1103}{576}$.
Likewise, if $O$ is a point in $F$, then
$$
S\big(W_{\bullet, \bullet,\bullet}^{\widetilde{S},F};O\big)=
\left\{\aligned
&\frac{131}{144}   \ \text{ if } O\in \widetilde{T}_P,\ \text{$T_P$ has a node at $P$},\ \text{and}\ O\in\widetilde{C}_5,\\
&\frac{29}{32}   \ \text{ if } O\in \widetilde{T}_P,\ \text{$T_P$ has a node at $P$},\ \text{and}\ O\not\in\widetilde{C}_5,\\
&\frac{277}{288}   \ \text{ if } O\in \widetilde{T}_P,\ \text{$T_P$ has a cusp at $P$},\ \text{and}\ O\in\widetilde{C}_5,\\
&\frac{23}{24}   \ \text{ if } O\in \widetilde{T}_P,\ \text{$T_P$ has a cusp at $P$}\ \text{and}\ O\not\in\widetilde{C}_5,\\
&\frac{1045}{1152}   \ \text{ if } O\in \widetilde{L}_1\cup \widetilde{L}_2\ \text{and}\ O\in\widetilde{C}_5, \\
&\frac{347}{384}  \ \text{ if } O\in \widetilde{L}_1\cup \widetilde{L}_2\ \text{and}\ O\not\in\widetilde{C}_5, \\
&\frac{247}{288}  \ \text{ if } O\not\in\widetilde{L}_1\cup\widetilde{L}_2\cup\widetilde{T}_P\ \text{and}\ O\in\widetilde{C}_5,\\
&\frac{41}{48}   \ \text{ if } O\not\in\widetilde{L}_1\cup\widetilde{L}_2\cup\widetilde{T}_P\ \text{and}\ O\not\in\widetilde{C}_5.
\endaligned
\right.
$$
The lemma is proved.
\end{proof}

\begin{lemma}
\label{lemma:cubic-surface-case-b}
Suppose $T_P=\ell+C_2$, where $\ell$ is a line, $C_2$ is an smooth conic such that $L_1\ne C_2\ne L_2$.
Then $S(W^S_{\bullet,\bullet};F)<2$ and $S(W_{\bullet, \bullet,\bullet}^{\widetilde{S},F};O)<1$ for every point $O\in F$.
\end{lemma}

\begin{proof}
Let $\widetilde{\ell}$ and $c$ be the proper transforms on the surface $\widetilde{S}$ of the curves $\ell$ and $C_2$, respectively.
Then $\widetilde{\ell}$ is a $(-2)$-curve, $\widetilde{C}_2$ is a $(-1)$-curve, and
the intersection $\widetilde{\ell}\cap \widetilde{C}_2$ consists of a single point.
Note also that $\widetilde{\ell}\cap \widetilde{C}_2\in F$ $\iff$  $\ell$ and $C_2$ are tangent at $P$.
If $0\leqslant u\leqslant\frac{2}{3}$, it follows from \eqref{equation:cubic-surface-blown-up} that
$$
\widetilde{P}(u,v)=
\left\{\aligned
&\widetilde{\ell}+\widetilde{C}_2+(1-u)\big(\widetilde{L}_1+\widetilde{L}_2\big)+(4-2u-v)F\ \text{for}\ 0\leqslant v\leqslant 3-2u,\\
&\frac{5-2u-v}{2}\widetilde{\ell}+\widetilde{C}_2+(1-u)\big(\widetilde{L}_1+\widetilde{L}_2\big)+(4-2u-v)F \ \text{for}\ 3-2u\leqslant v\leqslant 3-u,\\
&\frac{5-2u-v}{2}\widetilde{\ell}+\widetilde{C}_2+(4-2u-v)\big(\widetilde{L}_1+\widetilde{L}_2+F\big) \ \text{for}\ 3-u\leqslant v\leqslant \frac{11-6u}{3},\\
&(4-2u-v)\big(2\widetilde{\ell}+3\widetilde{C}_2+\widetilde{L}_1+\widetilde{L}_2+F\big) \ \text{for}\ \frac{11-6u}{3}\leqslant v\leqslant 4-2u,
\endaligned
\right.
$$
and
$$
\widetilde{N}(u,v)=
\left\{\aligned
&0\ \text{for}\ 0\leqslant v\leqslant 3-2u,\\
&\frac{v+2u-3}{2}\widetilde{\ell} \ \text{for}\ 3-2u\leqslant v\leqslant 3-u,\\
&\frac{v+2u-3}{2}\widetilde{\ell}+(v+u-3)\big(\widetilde{L}_1+\widetilde{L}_2\big) \ \text{for}\ 3-u\leqslant v\leqslant \frac{11-6u}{3},\\
&(4u+2v-7)\widetilde{\ell}+(6u+3v-11)\widetilde{C}_2+(v+u-3)\big(\widetilde{L}_1+\widetilde{L}_2\big) \ \text{for}\ \frac{11-6u}{3}\leqslant v\leqslant 4-2u,
\endaligned
\right.
$$
which gives
$$
\big(\widetilde{P}(u,v)\big)^2=
\left\{\aligned
&2u^2-v^2-12u+13\ \text{for}\ 0\leqslant v\leqslant 3-2u,\\
&\frac{35}{2}-18u-3v+4u^2-\frac{v^2}{2}+2uv \ \text{for}\ 3-2u\leqslant v\leqslant 3-u,\\
&\frac{71}{2}-30u-15v+6u^2+\frac{3v^2}{2}+6uv \ \text{for}\ 3-u\leqslant v\leqslant \frac{11-6u}{3},\\
&6(4-2u-v)^2 \ \text{for}\ \frac{11-6u}{3}\leqslant v\leqslant 4-2u,
\endaligned
\right.
$$
and
$$
\widetilde{P}(u,v)\cdot F=
\left\{\aligned
&v\ \text{for}\ 0\leqslant v\leqslant 3-2u,\\
&\frac{3-2u+v}{2} \ \text{for}\ 3-2u\leqslant v\leqslant 3-u,\\
&\frac{15-6u-3v}{2} \ \text{for}\ 3-u\leqslant v\leqslant \frac{11-6u}{3},\\
&24-12u-6v \ \text{for}\ \frac{11-6u}{3}\leqslant v\leqslant 4-2u,
\endaligned
\right.
$$
Likewise, if $\frac{2}{3}\leqslant u\leqslant 1$, then it follows from \eqref{equation:cubic-surface-blown-up} that
$$
\widetilde{P}(u,v)=
\left\{\aligned
&\widetilde{\ell}+\widetilde{C}_2+(1-u)\big(\widetilde{L}_1+\widetilde{L}_2\big)+(4-2u-v)F\ \text{for}\ 0\leqslant v\leqslant 3-2u,\\
&\frac{5-2u-v}{2}\widetilde{\ell}+\widetilde{C}_2+(1-u)\big(\widetilde{L}_1+\widetilde{L}_2\big)+(4-2u-v)F \ \text{for}\ 3-2u\leqslant v\leqslant \frac{11-6u}{3},\\
&(4-2u-v)\big(2\widetilde{\ell}+3\widetilde{C}_2+F\big)+(1-u)\big(\widetilde{L}_1+\widetilde{L}_2\big)\ \text{for}\ \frac{11-6u}{3}\leqslant v\leqslant 3-u,\\
&(4-2u-v)\big(2\widetilde{\ell}+3\widetilde{C}_2+\widetilde{L}_1+\widetilde{L}_2+F\big) \ \text{for}\ 3-u\leqslant v\leqslant 4-2u,
\endaligned
\right.
$$
and
$$
\widetilde{N}(u,v)=
\left\{\aligned
&0\ \text{for}\ 0\leqslant v\leqslant 3-2u,\\
&\frac{v+2u-3}{2}\widetilde{\ell} \ \text{for}\ 3-2u\leqslant v\leqslant \frac{11-6u}{3},\\
&(4u+2v-7)\widetilde{\ell}+(6u+3v-11)\widetilde{C}_2 \ \text{for}\ \frac{11-6u}{3}\leqslant v\leqslant 3-u,\\
&(4u+2v-7)\widetilde{\ell}+(6u+3v-11)\widetilde{C}_2+(v+u-3)\big(\widetilde{L}_1+\widetilde{L}_2\big) \ \text{for}\ 3-u\leqslant v\leqslant 4-2u,
\endaligned
\right.
$$
which gives
$$
\big(\widetilde{P}(u,v)\big)^2=
\left\{\aligned
&2u^2-v^2-12u+13\ \text{for}\ 0\leqslant v\leqslant 3-2u,\\
&\frac{35}{2}-18u-3v+4u^2-\frac{v^2}{2}+2uv \ \text{for}\ 3-2u\leqslant v\leqslant \frac{11-6u}{3},\\
&22u^2+20uv+4v^2-84u-36v+78 \ \text{for}\ \frac{11-6u}{3}\leqslant v\leqslant 3-u,\\
&6(4-2u-v)^2 \ \text{for}\ 3-u\leqslant v\leqslant 4-2u,
\endaligned
\right.
$$
and
$$
\widetilde{P}(u,v)\cdot F=
\left\{\aligned
&v\ \text{for}\ 0\leqslant v\leqslant 3-2u,\\
&\frac{3-2u+v}{2} \ \text{for}\ 3-2u\leqslant v\leqslant \frac{11-6u}{3},\\
&18-10u-4v+18 \ \text{for}\ \frac{11-6u}{3}\leqslant v\leqslant 3-u,\\
&24-12u-6v \ \text{for}\ 3-u\leqslant v\leqslant 4-2u,
\endaligned
\right.
$$
Similarly, if $1\leqslant u\leqslant\frac{4}{3}$, then
$$
\widetilde{P}(u,v)=
\left\{\aligned
&(4-3u)\big(\widetilde{\ell}+\widetilde{C}_2\big)+(8-6u-v)F\ \text{for}\ 0\leqslant v\leqslant 4-3u,\\
&\frac{12-9u-v}{2}\widetilde{\ell}+(4-3u)\widetilde{C}_2+(8-6u-v)F \ \text{for}\ 4-3u\leqslant v\leqslant\frac{20-15u}{2},\\
&(8-6u-v)\big(2\widetilde{\ell}+3\widetilde{C}_2+F\big) \ \text{for}\ \frac{20-15u}{2}\leqslant v\leqslant 8-6u,
\endaligned
\right.
$$
and
$$
\widetilde{N}(u,v)=
\left\{\aligned
&0\ \text{for}\ 0\leqslant v\leqslant 4-3u,\\
&\frac{v+3u-4}{2}\widetilde{\ell} \ \text{for}\ 4-3u\leqslant v\leqslant\frac{20-15u}{2},\\
&(9u+2v-12)\widetilde{\ell}+(15u+3v-20)\widetilde{C}_2\ \text{for}\ \frac{20-15u}{2}\leqslant v\leqslant 8-6u,
\endaligned
\right.
$$
which gives
$$
\big(\widetilde{P}(u,v)\big)^2=
\left\{\aligned
&27u^2-v^2-72u+48\ \text{for}\ 0\leqslant v\leqslant 4-3u,\\
&56-84u-4v+\frac{63u^2}{2}-\frac{v^2}{2}+3uv \ \text{for}\ 4-3u\leqslant v\leqslant\frac{20-15u}{2},\\
&4(6u+v-8)^2\ \text{for}\ \frac{20-15u}{2}\leqslant v\leqslant 8-6u,
\endaligned
\right.
$$
and
$$
\widetilde{P}(u,v)\cdot F=
\left\{\aligned
&v\ \text{for}\ 0\leqslant v\leqslant 4-3u,\\
&\frac{4-3u+v}{2} \ \text{for}\ 4-3u\leqslant v\leqslant\frac{20-15u}{2},\\
&32-24u-4v\ \text{for}\ \frac{20-15u}{2}\leqslant v\leqslant 8-6u,
\endaligned
\right.
$$
Now, we can compute
$$
S\big(W^S_{\bullet,\bullet};F\big)=\frac{1}{8}\int\limits_1^{\frac{4}{3}}(u-1)(27u^2-72u+48)du+\frac{1}{8}\int\limits_0^{\frac{4}{3}}\int\limits_0^{\widetilde{t}(u)}\big(\widetilde{P}(u,v)\big)^2dvdu=\frac{1661}{864}.
$$

Let $O$ be a point in $F$. If $O\in\widetilde{\ell}\cap\widetilde{C}_2\cap\widetilde{C}_5$, then $O\not\in\widetilde{L}_1\cup\widetilde{L}_2$, so that
\begin{multline*}
F_O\big(W_{\bullet,\bullet,\bullet}^{\widetilde{S},F}\big)=
\frac{1}{4}\int\limits_1^{\frac{4}{3}}\int\limits_0^{8-6u}(u-1)\big(\widetilde{P}(u,v)\cdot F\big)dvdu+
\frac{1}{4}\int\limits_0^{\frac{4}{3}}\int\limits_0^{\widetilde{t}(u)}\big(\widetilde{P}(u,v)\cdot F\big)\cdot \mathrm{ord}_O\big(\widetilde{N}(u,v)\big|_F\big)dvdu=\\
=\frac{1}{576}+\frac{1}{4}\int\limits_0^{\frac{2}{3}}\int\limits_{3-2u}^{\frac{11-6u}{3}}\big(\widetilde{P}(u,v)\cdot F\big)\frac{v+2u-3}{2}dvdu+\frac{1}{4}\int\limits_0^{\frac{2}{3}}\int\limits_{\frac{11-6u}{3}}^{4-2u}\big(\widetilde{P}(u,v)\cdot F\big)\big((6u+3v-11)+(4u+2v-7)\big)dvdu+\\
+\frac{1}{4}\int\limits_{\frac{2}{3}}^1\int\limits_{3-2u}^{\frac{11-6u}{3}}\big(\widetilde{P}(u,v)\cdot F\big)\frac{v+2u-3}{2}dvdu+\frac{1}{4}\int\limits_0^{\frac{2}{3}}\int\limits_{\frac{11-6u}{3}}^{4-2u}\big(\widetilde{P}(u,v)\cdot F\big)\big((6u+3v-11)+(4u+2v-7)\big)dvdu+\\
+\frac{1}{4}\int\limits_{1}^{\frac{4}{3}}\int\limits_{4-3u}^{\frac{20-15u}{3}}\big(\widetilde{P}(u,v)\cdot F\big)\frac{v+2u-3}{2}dvdu+\frac{1}{4}\int\limits_{1}^{\frac{4}{3}}\int\limits_{\frac{20-15u}{3}}^{8-6u}\big(\widetilde{P}(u,v)\cdot F\big)\big((6u+3v-11)+(4u+2v-7)\big)dvdu=\frac{235}{1728},
\end{multline*}
so that
$$
S\big(W_{\bullet, \bullet,\bullet}^{\widetilde{S},F};O\big)=\frac{1}{8}\int\limits_0^{\frac{4}{3}}\int\limits_0^{\widetilde{t}(u)}\big(\widetilde{P}(u,v)\cdot F\big)^2dvdu+\frac{235}{1728}=\frac{1685}{1728}.
$$
Similarly, if $O\in\widetilde{\ell}\cup\widetilde{C}_2$, then $S(W_{\bullet, \bullet,\bullet}^{\widetilde{S},F};O)\leqslant \frac{1685}{1728}$.
If $O\in\widetilde{L}_1\cup\widetilde{L}_2$, then $O\not\in\widetilde{\ell}\cup\widetilde{C}_2$,
and $O$ is contained in exactly one of the curves $\widetilde{L}_1$ or $\widetilde{L}_2$.
In this case, we have
$$
S\big(W_{\bullet, \bullet,\bullet}^{\widetilde{S},F};O\big)=
\left\{\aligned
&\frac{515}{576}   \ \text{ if } O\in\widetilde{C}_5,\\
&\frac{257}{288}   \ \text{ if } O\not\in\widetilde{C}_5.
\endaligned
\right.
$$
The lemma is proved.
\end{proof}

\begin{lemma}
\label{lemma:cubic-surface-case-e}
Suppose $T_P=\ell_1+\ell_2+\ell_3$, where $\ell_1$, $\ell_2$, $\ell_3$ are lines such that $P=\ell_1\cap\ell_2$ and $P\not\in\ell_3$.
Then $S(W^S_{\bullet,\bullet};F)<2$ and $S(W_{\bullet, \bullet,\bullet}^{\widetilde{S},F};O)<1$ for every point $O\in F$.
\end{lemma}

\begin{proof}
Let $\widetilde{\ell}_1$, $\widetilde{\ell}_2$, $\widetilde{\ell}_3$ be the~proper transforms on $\widetilde{S}$ of the~lines $\ell_1$, $\ell_2$, $\ell_3$, respectively.
If $0\leqslant u\leqslant 1$, then it follows from \eqref{equation:cubic-surface-blown-up} that
$$
\widetilde{P}(u,v)=
\left\{\aligned
&\widetilde{\ell}_1+\widetilde{\ell}_2+\widetilde{\ell}_3+(1-u)\big(\widetilde{L}_1+\widetilde{L}_2\big)+(4-2u-v)F \text{ for } 0\leqslant v\leqslant 3-2u, \\
&\frac{5-2u-v}{2}\big(\widetilde{\ell}_1+\widetilde{\ell}_2\big)+\widetilde{\ell}_3+(1-u)\big(\widetilde{L}_1+\widetilde{L}_2\big)+(4-2u-v)F   \ \text{ for } 3-2u \leqslant v\leqslant 3-u,\\
&\frac{5-2u-v}{2}\big(\widetilde{\ell}_1+\widetilde{\ell}_2\big)+\widetilde{\ell}_3+(4-2u-v)\big(\widetilde{L}_1+\widetilde{L}_2+F\big)   \ \text{ for } 3-u \leqslant v\leqslant 4-2u,
\endaligned
\right.
$$
and
$$
\widetilde{N}(u,v)=
\left\{\aligned
&0 \text{ for } 0\leqslant v\leqslant 3-2u, \\
&\frac{v+2u-3}{2}\big(\widetilde{\ell}_1+\widetilde{\ell}_2\big) \ \text{ for } 3-2u \leqslant v\leqslant 3-u,\\
&\frac{v+2u-3}{2}\big(\widetilde{\ell}_1+\widetilde{\ell}_2\big)+(v+u-3)\big(\widetilde{L}_1+\widetilde{L}_2\big) \ \text{ for } 3-u \leqslant v\leqslant 4-2u,
\endaligned
\right.
$$
which gives
$$
\big(\widetilde{P}(u,v)\big)^2=
\left\{\aligned
&2u^2-v^2-12u+13 \text{ for } 0\leqslant v\leqslant 3-2u, \\
&6u^2+4uv-24u-6v+22 \ \text{ for } 3-2u \leqslant v\leqslant 3-u,\\
&8u^2+8uv+2v^2-36u-18v+40 \ \text{ for } 3-u \leqslant v\leqslant 4-2u,
\endaligned
\right.
$$
and
$$
\widetilde{P}(u,v)\cdot F=
\left\{\aligned
&v \text{ for } 0\leqslant v\leqslant 3-2u, \\
&3-2u\ \text{ for } 3-2u \leqslant v\leqslant 3-u,\\
&9-4u-2v \ \text{ for } 3-u \leqslant v\leqslant 4-2u.
\endaligned
\right.
$$
Similarly, if $1\leqslant u\leqslant \frac{3}{2}$, then
$$
\widetilde{P}(u,v)=
\left\{\aligned
&(4-3u)\big(\widetilde{\ell}_1+\widetilde{\ell}_2+\widetilde{\ell}_{3}\big)+(8-6u-v)F  \ \text{ for } 0\leqslant v\leqslant 4-3u, \\
&\frac{12-9u-v}{2}\big(\widetilde{\ell}_1+\widetilde{\ell}_2\big)+(4-3u)\widetilde{\ell}_3+(8-6u-v)F \ \text{ for } 4-3u\leqslant v \leqslant 8-6u,
\endaligned
\right.
$$
and
$$
\widetilde{N}(u,v)=
\left\{\aligned
&0\ \text{ for } 0\leqslant v\leqslant 4-3u, \\
&\frac{v+3u-4}{2}\big(\widetilde{\ell}_1+\widetilde{\ell}_2\big) \ \text{ for } 4-3u\leqslant v \leqslant 8-6u,
\endaligned
\right.
$$
which gives
$$
\big(\widetilde{P}(u,v)\big)^2=
\left\{\aligned
&27u^2-v^2-72u+48  \ \text{ for } 0\leqslant v\leqslant 4-3u, \\
&2(4-3u)(8-6u-v) \ \text{ for } 4-3u\leqslant v \leqslant 8-6u,
\endaligned
\right.
$$
and
$$
\widetilde{P}(u,v)\cdot F=
\left\{\aligned
&v\ \text{ for } 0\leqslant v\leqslant 4-3u, \\
&4-3u\ \text{ for } 4-3u\leqslant v \leqslant 8-6u.
\endaligned
\right.
$$
Since $P\in C_5$ and $C_5$ is smooth,
we compute $S(W^S_{\bullet,\bullet};F)=\frac{31}{16}$.
Similarly, if $O$ is a point in $F$, then
$$
S\big(W_{\bullet, \bullet,\bullet}^{\widetilde{S},F};O\big)=
\left\{\aligned
&\frac{329}{384}   \ \text{ if } O\in \widetilde{\ell}_1\cup\widetilde{\ell}_{2},\\
&\frac{161}{192}   \ \text{ if } O\in \widetilde{L}_1\cup \widetilde{L}_2, \\
&\frac{155}{192}   \ \text{ if } O\not\in\widetilde{\ell}_1\cup\widetilde{\ell}_{2}\cup\widetilde{L}_1\cup\widetilde{L}_2.
\endaligned
\right.
$$
The lemma is proved.
\end{proof}

Thus, we see that $S(W^S_{\bullet,\bullet};F)<2$ and $S(W_{\bullet, \bullet,\bullet}^{\widetilde{S},F};O)<1$ for every point $O\in F$.
Hence, using  \eqref{equation:Kento-point}, we conclude that $\beta(\mathbf{F})>0$. The Main Theorem is proved.

\medskip
\noindent
\textbf{Acknowledgements}
We started this project during the first author's visit to Krakow in 2022,
and we completed it during the first author's second visits to Krakow in 2023.
We would like to cordially thank Pedagogical University of Krakow for the hospitality and excellent working conditions.

Cheltsov has been supported by EPSRC grant \textnumero EP/V054597/1 and by the Leverhulme Trust grant RPG-2021-229.
Pokora has been supported by The Excellent Small Working Groups Programme DNWZ.711/IDUB/ESWG/2023/01/00002 at the Pedagogical University of Krakow.


\begin{thebibliography}{99}

\bibitem{AbbanZhuang}
H.~Abban, Z.~Zhuang, \emph{K-stability of Fano varieties via admissible flags}, Forum of Mathematics, Pi \textbf{10} (2022), Paper No. e15, 43 p.

\bibitem{Book}
C.~Araujo, A.-M.~Castravet, I.~Cheltsov, K.~Fujita, A.-S.~Kaloghiros, J.~Martinez-Garcia, C.~Shramov, H.~S\"u\ss, N.~Viswanathan,
\emph{The Calabi problem for Fano threefolds}, LMS Lecture Notes in Mathematics \textbf{485}, Cambridge University Press, 2023.

\bibitem{ClarkLacy}
P.~Clark, A.~Lacy, \emph{There are genus one curves of every index over every infinite finitely
generated field}, J. Reine Angew. Math. \textbf{749} (2019), 65--86.

\bibitem{Fisher2001}
T.~Fisher, \emph{Some examples of $5$ and $7$ descent for elliptic curves over $\mathbb{Q}$},  J. Eur. Math. Soc. \textbf{3} (2001), 169--201.

\bibitem{Fisher2010}
T.~Fisher, \emph{Pfaffian presentations of elliptic normal curves}, Trans. Amer. Math. Soc. \textbf{362} (2010), 2525--2540.

\bibitem{Fujita2016}
K.~Fujita, \emph{On K-stability and the~volume functions of $\mathbb{Q}$-Fano varieties}, Proc. Lond. Math. Soc. \textbf{113} (2016), 541--582.

\bibitem{Fujita2019}
K.~Fujita, \emph{A valuative criterion for uniform K-stability of $\mathbb{Q}$-Fano varieties}, J. Reine Angew. Math. \textbf{751} (2019), 309--338.

\bibitem{Hartshorne}
R.~Hartshorne, \emph{Algebraic geometry}, Graduate Texts in Mathematics \textbf{52}, Springer--Verlag, 1983.

\bibitem{Hulek}
K.~Hulek, \emph{Projective geometry of elliptic curves}, Asterisque \textbf{137} (1986), 143 p.

\bibitem{Li}
C.~Li, \emph{K-semistability is equivariant volume minimization}, Duke Math. Jour. \textbf{166} (2017), 3147--3218.

\bibitem{ShinderLin}
H.-Y.~Lin, E.~Shinder, \emph{Motivic invariants of birational maps}, to appear in Annals of Math.

\bibitem{Nagata}
M.~Nagata, \emph{On self-intersection number of a section on a ruled surface}, Nagoya Math. J. \textbf{37} (1970), 191--196.

\bibitem{Zhuang}
Z.~Zhuang, \emph{Optimal destabilizing centers and equivariant K-stability}, Invent. Math. \textbf{226} (2021), 195--223.

\end{thebibliography}
\end{document}